\documentclass[a4paper]{article}

\usepackage[english]{babel}

\usepackage[utf8]{inputenc}
\usepackage[utf8]{inputenc}
\usepackage{amsmath}
\usepackage{graphicx}
\usepackage{amssymb}
\usepackage{amsthm}
\usepackage{tikz-cd}
\usepackage{mathrsfs}
\usepackage[colorinlistoftodos]{todonotes}
\usepackage{enumitem}
\usepackage{yfonts}
\usepackage{ dsfont }
\usepackage{MnSymbol}

\usepackage[affil-it]{authblk}

\title{Bubbling phenomenon for Hermitian Yang-Mills connections}

\author{Yang Li} 	

\makeatletter
\def\@maketitle{%
	\newpage
	\null
	\vskip 2em%
	\begin{center}%
		\let \footnote \thanks
		{\Large\bfseries \@title \par}%
		\vskip 1.5em%
		{\normalsize
			\lineskip .5em%
			\begin{tabular}[t]{c}%
				\@author
			\end{tabular}\par}%
		\vskip 1em%
		{\normalsize \@date}%
	\end{center}%
	\par
	\vskip 1.5em}
\makeatother

\newtheorem{thm}{Theorem}[section]
\newtheorem{lem}[thm]{Lemma}

\theoremstyle{definition}

\newtheorem{eg}[thm]{Example}

\newtheorem{cor}[thm]{Corollary}

\newtheorem*{rmk}{Remark}

\newtheorem*{Acknowledgement}{Acknowledgement}

\newcommand{\cf}{\emph{cf.} }

\newcommand{\R}{\mathbb{R}}
\newcommand{\C}{\mathbb{C}}
\newcommand{\Z}{\mathbb{Z}}

\newcommand{\Lap}{\Delta}
\newcommand{\norm}[1]{\left\lVert #1 \right\rVert}

\begin{document}
\maketitle

\begin{abstract}
We construct local examples of singular Hermitian Yang-Mills connections over $B_1\subset \C^3$ with uniformly bounded $L^2$-energy, but the number of essential singular points can be arbitrarily large.
\end{abstract}

A central goal in the analytic study of higher dimensional gauge theory is to find a good compactification of the moduli spaces of instantons, or more generally Yang-Mills (YM) connections. The best known general analytic convergence theory \cite{Tian}\cite{Tao}\cite{Naber} in the local setting can be summarized as follows. Given a sequence of YM connections $A_i$ on $B_2\subset \R^n$ with $n\geq 4$, assuming the uniform bound on $L^2$ energy $\int |F_{A_i}|^2 dvol\leq \Lambda$, then after passing to subsequences, the followings hold true on an interior ball $B_1$:

\begin{itemize}
\item (Codimension 4 convergence \cite{Tian}) Away from a possibly empty $(n-4)$-rectifiable closed subset $\Sigma$ with $\mathcal{H}^{n-4}(\Sigma)\leq C(n,\Lambda)$, modulo gauge transforms $A_i$ converge in $C^\infty_{loc}$ topology to a limiting smooth YM connection $A_\infty$ on $B_1\setminus \Sigma$. If $A_i$ are higher dimensional instantons (eg. Hermitian-Yang-Mills connections, $G_2$-instantons,  or $Spin(7)$-instantons), then so is $A_\infty$, and $\Sigma$ is a calibrated current (resp. complex codimension 2 subvarieties, associative currents, Cayley currents).

\item  (Removable singularity \cite{Tao}\cite{Uhlenbeck2}) Assume further that $A_\infty$ is stationary on $B_1$. There is a possibly empty subset $Sing\subset \Sigma$ with $\mathcal{H}^{n-4}(Sing)=0$, such that after gauge transform, $A_\infty$ extends to a smooth connection on $B_1\setminus Sing$. The smallest such $Sing$ is called the \emph{essential singular set} of $A_\infty$.

\item
(Transverse bubbles \cite{Tian}\cite{Naber})  A transverse bubble $A_x$ at $x\in \Sigma$ is a smooth Yang-Mills connection over $\R^n$ isomorphic to the pullback of a YM connection on $\R^4\simeq \R^n/T_x\Sigma$, such that for some sequence of centre points $x_i\to x$ and scale parameters $\epsilon_i$ depending on $x_i$, the blow up sequence $\lambda_{ x_i, \epsilon_i }^* A_i$ subsequentially converge to $A_x$ up to gauge transform, where $\lambda_{ x_i, \epsilon_i }$ refers to the scaling diffeomorphism $y\mapsto  x_i+\epsilon y$. Nontrivial transverse bubbles exist at $\mathcal{H}^{n-4}$-a.e $x\in \Sigma$. If $A_i$ are higher dimensional instantons, then $A_x$ are isomorphic to the pullback of ASD connections on $\R^4$.

\item  (Energy identity \cite{Tian}\cite{Naber}) The curvature density converges weakly as measures:
\[
|F_{A_i}|^2 dvol\to |F_{A_\infty}|^2 dvol +\nu, \quad \nu= \Theta \mathcal{H}^{n-4}\lfloor _\Sigma ,
\]
where $\nu$ is a measure supported on $\Sigma$ (called the defect measure), and the density $\Theta: \Sigma \to \R_+$ is a bounded function. For $\mathcal{H}^{n-4}$-a.e. $x\in \Sigma$, the function $\Theta(x)$ can be computed as the sum of $L^2$-energies associated with the transverse bubble connections, where one needs to account for multiple bubbling phenomenon.

\item  (Curvature Hessian estimate \cite{Naber}) $\int_{B_1} |\nabla_{A_i}^2 F_{A_i}| \leq C(n, \Lambda)$.
\end{itemize}

The data $(A_\infty, \Sigma, \Theta)$ are referred to as \emph{ideal instantons}; including these into the instanton moduli space results in the 
\emph{Tian compactification}. This theory works on compact manifolds essentially verbatim. The fact that the $L^2$-energy is conformally invariant in dimension 4, is the ultimate reason why codimension 4 is critical in the above theory.

Conversely, Walpuski \cite{Walpuski} and others have  performed gluing constructions of higher dimensional instantons on manifolds with special holonomy, exhibiting the codimension 4 bubbling phenomenon. A sample gluing theorem requires the following data:
\begin{itemize}
\item  A \emph{smooth} instanton connection $A_\infty$ (eg. Hermitian Yang-Mills connection, $G_2$-instanton, $Spin(7)$-instanton) on a principal $G$-bundle $P_\infty$ over a (compact) special holonomy manifold $M$, satisfying some irreducibility and unobstructedness conditions.

\item  A \emph{smooth} calibrated  codimension 4 submanifold $\Sigma\subset M$, satisfying an unobstructedness condition.

\item  A principal $G$-bundle $P'$ over the normal bundle $\mathcal{N}\Sigma$ of $\Sigma$, with framing data so that over tubular neighbourhoods of $\Sigma$ one can identify $P_\infty$ with $P'$. For each normal plane $\mathcal{N}_x\Sigma\simeq \R^4$ one takes the  moduli space $\mathcal{M}_x$ of (framed) ASD instantons on $P'|_{\mathcal{N}_x\Sigma }$, and fits them together into a moduli bundle $\mathcal{M}\to \Sigma$ as $x\in \Sigma$ varies.

\item  A section $s$ of $\mathcal{M}\to \Sigma$ satisfying a first order Dirac type equation called the \emph{Fueter equation}. (For example, in the Hermitian Yang-Mills case this means $s$ is holomorphic). This is required to satisfy some non-degeneracy conditions.
\end{itemize}

The output is a 1-parameter family $\{A_\epsilon\}_{0<\epsilon\ll 1}$ of instantons over $M$ with uniformly bounded $L^2$-energy, which $C^\infty_{loc}$-converge to $A_\infty$ away from $\Sigma$ as $\epsilon\to 0$, and the $L^2$-curvatures bubble along $\Sigma$. Morever, the transverse bubbles are controlled by the Fueter map $s$, in the sense that for every $x\in \Sigma$, the scaled connections $\lambda_{x,\epsilon}^* A_\epsilon |_{ \mathcal{N}_x\Sigma}$ converge to a (framed) ASD connection on $\mathcal{N}_x\Sigma$ whose moduli paramter is $s|_x$. Intuitively, the Fueter map asserts large scale regularity for the variation of the transverse bubbles.

The fundamental problem in the compactification theory is that gluing methods require much stronger regularity assumptions of gluing data than would be waranted by the analytic convergence theory \emph{a priori}.  Standard open problems in this vein include the codimension conjecture (\cf \cite[Conjecture 2]{Tian}):
\begin{itemize}
\item Given any stationary Yang-Mills connection, show that the dimension of the essential singular set has Hausdorff dimension at most $n-5$. For nearly K\"ahler instantons, $G_2$-instantons and $Spin(7)$-instantons, improve this dimension bound to $n-6$.
\end{itemize}

This paper makes no positive progress on the compactification problem. Instead, we provide local examples illustrating \emph{why} it is difficult:

\begin{thm}\label{introthm}
Given every $N\in \mathbb{N}$ and $0<\epsilon< \epsilon_0(N)$, 
there exists a singular Hermitian Yang-Mills (HYM) metric  $h_{N,\epsilon}$ over $B_1\subset \C^3$, such that the $L^2$-energy is uniformly bounded independent of $\epsilon$ and $N$, but the number of essential singularities in $B_{1/2}$ is proportional to $N$. 
\end{thm}

Thus if we assume only a given bound on the $L^2$-energy of a stationary YM connection, then there is \emph{no a priori quantitative bound} on the $(n-6)$-Hausdorff measure of the  singular set, and in particular no bound on the topological complexity of the bundle. The rough idea is that for any fixed $N$ and $0<\epsilon < \epsilon_0(N)$, we can construct HYM connections on $B_1\subset \C^3$ starting from a Fueter map with very high oscillation:
\[
\{ |z|<1  \}\subset \C\to \C^2/\Z_2, \quad z\to (\sin (Nz)^{1/2}, 0), 
\]
such that the curvature is concentrated near $\Sigma=\{  x=y=0 \}$, and the key point is to control the $L^2$-energy uniformly in $N$ and $\epsilon$. The singular points correspond to the zeros of the Fueter maps, namely $\{ x=y=0, z\in \pi N^{-1} \Z  \}\cap B_1 \subset \C^3$. (Here $\C^2/\Z_2$ is the moduli space of framed one-instantons centred at the origin, see Example \ref{ADHMconstruction}.)

This high oscillation phenomenon also means that the Fueter maps satisfy no uniform estimate on the modulus of continuity.
The weak limit for the above sequence of Fueter maps is zero, which indicates a failure of compactness. Notably, if we take a subsequence of HYM metrics $\{ h_{N, \epsilon(N)} \}_{N=1}^\infty$ with $\epsilon(N)\to 0$ very rapidly as $N\to \infty$, then there is no way to extract a nontrivial limiting Fueter map as $N\to \infty$.

Since the $L^2$-energy is the only a priori controlled quantity which has found effective analytic use in general YM theory, these examples pose challenges to the compactification problem.

\begin{rmk}
On the positive side, it could be hoped that on a \emph{compact} manifold with special holonomy, the bad behaviours of our examples do not actually occur:
\begin{itemize}
	\item
It is known how to construct good compactifications of moduli spaces of HYM connections of a fixed Chern character over \emph{projective} manifolds, via the Hitchin-Kobayashi principle \cite{Ben}. This crucially uses the \emph{boundedness} of the corresponding family of stable  vector bundles, to prevent the wild behaviour of the holomorphic structure. The open challenge then is to understand this boundedness analytically.
\item
One may hope that such examples never arise as the limit of \emph{smooth} YM connections. This should be understood without recourse to complex geometry.
\item
Our examples have very high topological complexity. 
It is possible that such examples are ruled out by suitable \emph{genericity} hypotheses.
\end{itemize}
 
\end{rmk}

\begin{Acknowledgement}
The author is a postdoc at the IAS, funded by the Zurich Insurance Company Membership.
He thanks  Simon Donaldson, Thomas Walpuski and Aleksander Doan for  discussions.
\end{Acknowledgement}

\section{A model HYM metric}\label{monadconstruction}

We begin with a general curvature formula for the cohomology of a monad over a complex manifold.

\begin{lem}\label{curvatureformulamonadlemma}
Consider a monad   $E_0\xrightarrow{\alpha} E_1\xrightarrow{\beta} E_2$, namely a complex of Hermitian holomorphic vector bundles with $\alpha$ injective fibrewise and $\beta$ surjective fibrewise. Let $E=\ker \beta/\text{coker}( \alpha )$ be the cohomology bundle. Then the curvature $F_E$ of the natural induced connection on $E$ satisfies
\[
\langle F_E s, s'\rangle= 
\langle F_{E_1} s, s'\rangle
- \langle (\beta \beta^\dag)^{-1} (\nabla \beta) s, (\nabla \beta)s' \rangle
- \langle (\alpha^\dag \alpha)^{-1} (\nabla \alpha^\dag) s, (\nabla \alpha^\dag)s' \rangle,
\]
where $F_{E_1}$ is the Chern connection on $E_1$, and $s,s'$ are representing smooth sections of $E$ satisfying $\alpha^\dag s= \alpha^\dag s'= \beta s= \beta s'=0$, and $\nabla \alpha^\dagger$, $\nabla \beta$ are covariant derivatives computed on the Hom bundles.
\end{lem}

\begin{eg}\label{ADHMconstruction}
(ADHM construction of  one-instantons) Start from the monads over Euclidean $\C^2_{x,y}$ 
\[
\underline{\C} \xrightarrow{  (x, y, a_1, a_2)^t }\underline{\C^4} \xrightarrow{ (-y,x,b_1, b_2) }\underline{\C},
\]
where the underlines signify trivial vector bundles, and the parameters $a_1, a_2, b_1, b_2$ satisfy the ADHM equation
\[
a_1b_1+ a_2b_2=0, \quad |a_1|^2+ |a_2|^2= |b_1|^2+ |b_2|^2>0.
\]
The natural connections on cohomology bundles $E_{a,b}$ are ASD instantons on $\C^2$ with rank 2, charge 1 and curvature scale $\sim \sqrt{ |a_1|^2+|a_2|^2}$. The situation with $a_1=a_2=b_1=b_2=0$ is viewed as a degenerate case.
As Hermitian vector bundles $E_{a,b}$ are identified as $\ker \beta \cap \ker \alpha^\dagger\subset \underline{\C^4}=\underline{\C^2}\oplus \underline{\C^2} $, and the monad construction provides a natural projection map into the second $\underline{\C^2} $ factor, giving a trivialisation of $E_{a,b}$ near infinity known as \emph{framing data}. Notice the framed instantons are isomorphic under the $U(1)$-symmetry
\[
(a_1,a_2,b_1,b_2)\mapsto (a_1 e^{i\theta}, a_2 e^{i\theta}, b_1 e^{-i\theta},b_2 e^{-i\theta} ).
\]
The \emph{moduli space of framed instantons} centred at the origin is 
\[
\{ a_1b_1+a_2b_2=0, |a_1|^2+ |a_2|^2= |b_1|^2+ |b_2|^2     \}/U(1) \simeq \{ a_1b_1+a_2b_2=0 \}/\C^* \simeq \C^2_{u,v}/\Z_2.
\]
The last identification is explicitly 
\[
uv= a_1 b_1=- a_2b_2, \quad u^2= a_1b_2, \quad v^2=-a_2b_1.
\]
\end{eg}

In the rest of this section we describe a model HYM connection over $\C^3_{X,Y,Z}$ constructed by the author in \cite{Li}. The underlying reflexive sheaf $\mathcal{E}$ is the cohomology of the monad
\begin{equation}\label{modelmonad}
\underline{\C} \xrightarrow{ A=(X,Y,  1, 0)^t } \underline{\C^4}  \xrightarrow{ B= (-Y,X, 0,  Z)  }  \underline{\C},
\end{equation}
which is isomorphic to the homogeneous sheaf $\ker (\underline{\C^3}\xrightarrow{(X,Y,Z)} \underline{\C}  )
$, with a unique singular point at the origin. Using the Euler sequence over $\mathbb{P}^2$
\[
0\to \Omega_{\mathbb{P}^2 } \to \mathcal{O}(-1)^{\oplus 3} \to  \mathcal{O} \to 0,
\]
$\mathcal{E}$ is isomorphic as vector bundles over $\C^3\setminus \{0\}$ to the pullback of $\Omega_{\mathbb{P}^2} $ via $\C^3\setminus \{0\}\to\mathbb{P}^2 $.

An explicit background Hermitian metric $H_0$ on $\mathcal{E}$ is induced by
equipping  $\underline{\C^4}$ with
the \emph{nonstandard Hermitian structure}  given by the diagonal matrix 
\[
H_{\C^4}=\text{diag}( |\vec{X}|^{-1} , |\vec{X}|^{-1},   1,1 )
\]
where $|\vec{X}|^2=|X|^2+|Y|^2+|Z|^2$.

\begin{rmk}
The metric  $\text{diag}( (|\vec{X}|^2+1)^{-1/2} , (|\vec{X}|^2+1)^{-1/2},   1,1 )$ is used in \cite{Li}, but since we only need this in the region $\{ |\vec{X}|\gtrsim 1  \}$, these two choices are equivalent for the sake of estimates.
\end{rmk}

\begin{thm}\cite[Thm 2.2]{Li}
	There is a HYM connection $H$ on $\mathcal{E}$ over Euclidean $\C^3$ with locally finite $L^2$ curvature, which admits the decay estimates for $|\vec{X}|\gtrsim 1$:
	\begin{equation}\label{asymptoticestimate}
	|\nabla^k_{H_0}  \log ( H H_0^{-1}   ) |\lesssim_k ( |X|+|Y|+|Z|^{1/2}    )^{-k} |\vec{X}|^{-1} \max(1, \log \frac{|\vec{X}|}{|X|+|Y|+|Z|^{1/2}}  ), \quad k\geq 0.
	\end{equation}
The asymptotic estimate (\ref{asymptoticestimate}) and the HYM condition $\Lambda F_H=0$ determine $H$ uniquely.
\end{thm}

Morever, the tangent cone of $H$ at the origin agrees with the pullback of the Levi-Civita connection on $\Omega_{\mathbb{P}^2} $, up to a conformal change of the Hermitian structure, and $H$ is asymptotic to this tangent cone at a polynomial rate.
In particular,
\begin{equation}\label{FHbound1}
|\nabla^k_H  F_H| \lesssim_k |\vec{X}|^{-2-k}, \quad |\vec{X}|\lesssim 1, \quad k\geq 0.
\end{equation}

We now summarize the asymptotic behaviour of $H$ for $|\vec{X}|\gg 1$. The curvature satisfies a decay estimate
\begin{equation}\label{FHbound2}
|\nabla^k_H F_H| \lesssim_k \frac{ |\vec{X}|}{ (|X|+|Y|+|Z|^{1/2})^{4+k}  }, \quad |\vec{X}|\gtrsim 1, \quad k\geq 0.
\end{equation}
In particular, in the generic region where $|X|+|Y|$ is comparable to $|Z|$, the curvature has cubic decay rate as $|\vec{X}|\to \infty$, which entails that \emph{the tangent cone at infinity is flat}. The regularity scale around a point in $\{ |\vec{X}|\gtrsim 1  \}$ is at least $|X|+|Y|+|Z|^{1/2}$.

The curvature concentrates near $\{ |X|+|Y|\lesssim |Z|^{1/2} \}$. Take a ball of size $O(|\zeta|^{1/2})$ around a point $(0,0,\zeta)$ in this region with $|\zeta|\gg 1$. Observe that if the ambient Hermitian metric $H_{\C^4}$ is changed into 
\[
|\vec{X}| H_{\C^4}= \text{diag} (  1,1, |\vec{X}|, |\vec{X}|    ),
\]
then the induced connection on $\mathcal{E}$ is twisted by a $U(1)$ connection. We choose a square root $\zeta^{1/2}$. After rescaling the basis vectors on $\underline{\C^4}$, the monad (\ref{modelmonad}) can be written locally as
\[
\underline{\C} \xrightarrow{ (X,Y, \zeta^{1/2}, 0 )^t } \underline{\C^4} \xrightarrow{ (-Y, X, 0 ,\zeta^{1/2} )} \underline{\C},
\]
and the twisted Hermitian metric $|\vec{X}| H_{\C^4}$ is 
\[
\text{diag}( 1,1, \frac{ |\vec{X}|}{ |\zeta|},  \frac{ |\zeta| |\vec{X}|} { |Z|^2 }    ) \approx \text{diag}( 1,1, 1,  1    ).
\]
Comparing with Example \ref{ADHMconstruction}, we see that the twisted HYM metric on $\mathcal{E}$ locally approximately dimensionally reduces to a framed ADHM one-instanton, whose moduli parameter is identified as $(\zeta^{1/2},0)\in \C^2/\Z_2$.

\section{Main ansatz}

Our \emph{main ansatz} is conceptually a periodic version of the HYM connection described in section \ref{monadconstruction}. Let $N\in \mathbb{N}$, and $0<\epsilon \ll 1$ depending on $N$; all dependence on $\epsilon$ and $N$ will be explicitly tracked down.  For technical convenience, we will make the construction over a very large 
domain
\begin{equation}
\mathcal{D}_N=\{  |x|+|y|<N, \quad |\text{Im}(z)|<N   \}\subset \C^3/\Z.
\end{equation}
equivalently viewed as a domain in $\C^3$ periodic with respect to $z\mapsto z+2\pi$; later we will scale the solution down to the unit ball.  Consider the monad
\begin{equation}\label{mainansatzmonad}
\underline{\C} \xrightarrow{ \alpha=(x,y, \epsilon^2, 0)^t } \underline{\C^4}  \xrightarrow{ \beta= (-y,x, 0, \sin z)  }  \underline{\C},
\end{equation}
and take $E$ as the cohomology sheaf $\ker \beta/ \text{Im} \alpha$. We may assume $|\text{Re}(z)|\leq \pi$. Take a  function $r\geq 0$ on $\C^3/\Z$ with the properties:
\begin{itemize}
\item  In $\{ |\text{Im}(z)|>1  \}$, the function $r=|\sin z|$. The reader should keep in mind that $\sin z$ is a complex valued function, which is exponentially large for $|\text{Im}(z)|\gg 1$.
\item  In $\{ |\vec{x}|= \sqrt{|x|^2+|y|^2+|z|^2}  \gtrsim \frac{1}{3} ,   |\vec{x}-(0,0,\pi)|\gtrsim \frac{1}{3}, |\text{Im}(z)|\lesssim 1   \}$, then 
\[
|r^2-|\sin z|^2  |\leq C (|x|^2+|y|^2), \quad \frac{1}{10} \leq r\leq 10, \quad |\frac{\partial r}{\partial x}|+|\frac{\partial r}{\partial y}| \leq C(|x|+|y|),
\]
and $r$ has smooth bounds $\norm{r}_{C^k}\leq C(k)$. 
\item 
In $\{ |\vec{x}| \leq \frac{1}{3}    \}$, then $r=\sqrt{ |x|^2+|y|^2+|\sin z|^2}$. Similarly with the neighbourhood of the other singular point $(0,0,\pi)$.

\end{itemize}

 We provisionally equip $\underline{\C^4}$ with the nonstandard Hermitian structure given by the diagonal matrix
\[
h_{\C^4 }= \text{diag} ( \epsilon^2 r^{-1}, \epsilon^2 r^{-1}, 1,1       ), 
\]
which induces a background Hermitian  metric $h_0$ on $E$. The point is that $h_0$ has concentrated curvature near the $z$-axis, and is approximately HYM away from a very small neighbourhood of the origin.
A useful quantity to measure curvature concentration is
\begin{equation}
\ell= |x|+|y|+ \epsilon |\sin z|^{1/2}.
\end{equation}

\begin{lem}\label{meancurvaturebound}
The mean curvature $\Lambda F_{h_0}$ of the induced connection on $E$ admits the estimate on $\mathcal{D}_N$
\[
|\Lambda F_{h_0}| \leq  \begin{cases}	C \epsilon^2 r\ell^{-2} ,\quad  & r\gtrsim \frac{1}{3},
	\\
	C\epsilon^2 r^{-1}\ell^{-2}  , \quad &\epsilon^2\lesssim   r < \frac{1}{3}.
	\end{cases}
	\]
The curvature itself satisfies
\[
|F_{h_0} | \leq C \epsilon^2 r \max( \ell^{-4}, \ell^{-2}),\quad    r \gtrsim \epsilon^2. 
\]
\end{lem}

\begin{proof}
	We  shall compute the curvature $F_{h_0}$ using Lemma \ref{curvatureformulamonadlemma}. Taking into account the nonstandard Hermitian structure, the adjoint maps are given by
	\[
	\alpha^\dagger=( \epsilon^2\bar{x}r^{-1}, \epsilon^2\bar{y}r^{-1} , \epsilon^2, 0  ), \quad \beta^\dagger= (- \epsilon^{-2}\bar{y}r, \epsilon^{-2}\bar{x}r, 0,  \overline{ \sin z }  )^t,
	\]
	hence 
	\begin{equation}\label{betabetadagger}
	\alpha^\dag \alpha= \epsilon^2(|x|^2+|y|^2) r^{-1}+ \epsilon^4  , \quad \beta\beta^\dagger= \epsilon^{-2} (|x|^2+|y|^2)r    +  |\sin z|^2 .
	\end{equation}
For $r\gtrsim \epsilon^2$, then $\alpha^\dagger\alpha$ is comparable to $\epsilon^2 \ell^2 r^{-1}$, while $\beta\beta^\dagger$ is comparable to $\epsilon^{-2}\ell^2 r$.

Since $\alpha, \beta$ are holomorphic, $\alpha^\dagger, \beta^\dag $ are antiholomorphic, 
\begin{equation*}
	\begin{cases}
	\nabla \alpha^\dagger=\bar{\partial} \alpha^\dagger=   ( \epsilon^2 r^{-1} d\bar{x}-   \frac{ \epsilon^2 \bar{x} \bar{\partial}r }{ r^2  } , r^{-1} d\bar{y}- \frac{ \epsilon^2 \bar{y}\bar{\partial} r }{ r^2  }, 0, 0),
	\\
	\nabla \beta=  ( \bar{\partial} \beta^\dag  )^\dagger= ( -dy - y  r^{-1}\partial r
  , 
 dx+ x r^{-1}\partial r , 0, \cos z dz    ).
\end{cases}
\end{equation*}
Let $s=(s_1, s_2,s_3,s_4)^t$ be a smooth local section of $E$ represented as a section of $\underline{\C^4}$ with $\beta s= \alpha^\dag s=0$, so
\begin{equation}\label{nablaalphabeta}
\begin{cases}
(\nabla \beta)s= -s_1 dy+ s_2 dx-s_4\sin (z)  r^{-1}\partial r + s_4 \cos z dz,
\\
(\nabla \alpha^\dag) s= s_1 \epsilon^2 r^{-1} d\bar{x}+ s_2 \epsilon^2 r^{-1} d\bar{y} + s_3\epsilon^2 r^{-2} \bar{\partial} r .
\end{cases}
\end{equation}

By expressing $s_1, s_2$ in terms of $s_3, s_4$,
	\[
	|s_1|+|s_2| \leq \frac{ C r (|s_3|+|s_4|)  }{ |x|+|y|  } .
	\]
	Under the Hermitian structure
	\[
	|s|_{h_{\C^4}}^2= \epsilon^2(|s_1|^2+ |s_2|^2 )r^{-1} +  (|s_3|^2+|s_4|^2),
	\]
	we have
	\begin{equation}\label{s1+s2}
	|s_1|+|s_2|\leq C\min\{ \epsilon^{-1}r^{1/2}, \frac{  r }{ |x|+|y|  } \} |s|_{h_{\C^4}}.
	\end{equation}
	Now the Chern curvature on $\underline{\C^4}$ is given by $F_{E_1}= \bar{\partial} ( \partial h_{\C^4} h_{\C^4}^{-1} ),
	$
	so for $|\text{Im}(z)|> 1$,
	\[
	\langle \sqrt{-1} F_{E_1} s, s \rangle = 0,
	\]
and for $|\text{Im}(z)|\leq 1, r\gtrsim \epsilon^2$,
\[	
	\langle \sqrt{-1} F_{E_1} s, s \rangle	
	\leq C \epsilon^2 r^{-3}
(|s_1|^2+|s_2|^2  ) \leq C \epsilon^2  \frac{ |s|_{h_{\C^4}}^2}{ r\ell^2   } .
	\]	
Substituting (\ref{betabetadagger})(\ref{nablaalphabeta})(\ref{s1+s2}) into the curvature formula in Lemma \ref{curvatureformulamonadlemma}, 	
\begin{equation}\label{curvatureansatz}
\begin{split}
\langle \sqrt{-1}\Lambda F_{h_0} s,s\rangle =& -(\alpha^\dag \alpha)^{-1}  \epsilon^4 r^{-2} \sqrt{-1} \Lambda(s_1 d\bar{x}+s_2d\bar{y})\wedge (\bar{s}_1 d{x}+\bar{s}_2d{y})
\\
&- (\beta \beta^\dag)^{-1} \sqrt{-1} \Lambda(-s_1 dy+s_2dx)\wedge (-\bar{s}_1 d\bar{y}+\bar{s}_2d\bar{x})
\\
&+ \begin{cases}
O( \epsilon^2 r \ell^{-2} |s|_{h_{\C^4}}^2  ), \quad r\gtrsim \frac{1}{3},
\\
O(  \epsilon^2 \ell^{-2} r^{-1} |s|_{h_{\C^4}}^2  ), \quad \epsilon^2 \lesssim r\leq \frac{1}{3}.
\end{cases}
\end{split}
\end{equation}
Now notice the following cancellation effect: 
\begin{equation}\label{cancellationeffect}
|\epsilon^4r^{-2}(\alpha^\dag \alpha)^{-1}- (\beta \beta^\dag)^{-1}| \begin{cases}
=0  \quad |\text{Im}(z)|>1,
\\
\lesssim  \epsilon^4 \ell^{-2} , \quad  |\text{Im}(z)|\lesssim 1, r\gtrsim \frac{1}{3},
\\
\lesssim \epsilon^4 \ell^{-2}r^{-2}, \quad  \epsilon^2\lesssim  r < \frac{1}{3}.
\end{cases}
\end{equation}
hence 
\[
|\langle \sqrt{-1}\Lambda F_{h_0} s,s\rangle| \leq  \begin{cases}
C \epsilon^2 r\ell^{-2} |s|_{h_{\C^4}}^2,\quad  r\gtrsim \frac{1}{3},
\\
C\epsilon^2 r^{-1}\ell^{-2} |s|_{h_{\C^4}}^2 , \quad \epsilon^2\lesssim   r < \frac{1}{3},
\end{cases} 
\]
whence the required estimate on $|\Lambda F_{h_0}|$ follows. The estimate on $|F_{h_0}|$ proceeds the same way without using the cancellation effect.
\end{proof}

We now modify the Hermitian structure in a small neighbourhood of the origin, by gluing in the model HYM metric $H$ over $\C^3$ (\cf section \ref{monadconstruction}). 
Upon the variable substituion
\begin{equation}\label{gluingmap}
x= \epsilon^2 X, \quad  y= \epsilon^2 Y, \quad \sin z= \epsilon^2 Z,
\end{equation}
then $E$ pulls back to $\mathcal{E}$ (since modifying $\alpha, \beta$ by constants does not change the cohomology bundle), 
the Hermitian metric $h_{\C^4}$ is identified with $H_{\C^4}$, and $h_0$ with $H_0$. We can glue $h_0$ to $H$ on a neck region to obtain an approximate HYM metric $h_1$ on $E$:
\[
h_1= h_0+ \eta( H-h_0 ),
\]
where $\eta$ is a bump function supported on $\{r< 2\epsilon^{2/3} \}$, equal to one on $\{  r \leq \epsilon^{2/3}  \}$.

\begin{lem}\label{meancurvaturebound2}
For $r\lesssim 2\epsilon^{2/3}$, the mean curvature $\Lambda F_{h_1}$ satisfies the estimate
\begin{equation}
| \nabla^k_{h_1}  \Lambda F_{h_1} | \lesssim_k  \begin{cases}
r^{-k}, \quad   & r\lesssim  \epsilon^2,
\\
 r^2 \ell^{-2-k}     , \quad   & \epsilon^2< r\leq \epsilon^{2/3},
 \\
 \epsilon^2 \ell^{-2-k} r^{-1} \max(1, \log \frac{r}{\ell} )   , \quad   & \epsilon^{2/3}< r\leq 2\epsilon^{2/3}.
\end{cases}
\end{equation}
The curvature itself satisfies
\begin{equation}\label{curvatureHrescaled}
|\nabla^k_{h_1}  F_{h_1}| \lesssim_k \begin{cases}
\epsilon^2 r \ell^{-4-k}, \quad & \epsilon^2< r\leq 2\epsilon^{2/3}
\\
r^{-2-k}, \quad & r\lesssim \epsilon^2.
\end{cases}
\end{equation}
\end{lem}

\begin{proof}
Using the Taylor expansion of $\sin z$ near $z=0$, the Euclidean metrics 
\[
\omega= \frac{\sqrt{-1}}{2}(dx\wedge d\bar{x}+ dy\wedge d\bar{y}+ dz\wedge d\bar{z}), \quad  \omega'= \frac{\sqrt{-1}}{2}(dX\wedge d\bar{X}+ dY\wedge d\bar{Y}+ dZ\wedge d\bar{Z})
\]	
are approximately conformal for $|\vec{x}|\ll \frac{1}{3}$:
	\[
	|\epsilon^{4}\omega'- \omega|_{\omega} \leq C|z|^2 .
	\]
The curvature estimates (\ref{FHbound1})(\ref{FHbound2}) then translate into (\ref{curvatureHrescaled}) by scaling.
The $\omega'$-HYM metric $H$ is approximately HYM with respect to $\omega$ as well:
	\begin{equation}\label{meancurvatureestimatepreparation}
	\begin{split}
	|  \Lambda F_H| &\lesssim |F_H|_\omega | \epsilon^{4}\omega'- \omega|_\omega \lesssim |z|^2 |F_H|_\omega \lesssim \epsilon^{-4}|z|^2 |F_H|_{\omega'}
\\
	\lesssim & |z|^2
	\begin{cases}
	\epsilon^2 r \ell^{-4}, \quad  \epsilon^2< r\leq 2\epsilon^{2/3}
	\\
	r^{-2}, \quad  \quad\quad r\lesssim \epsilon^2
	\end{cases}
	\\
	\lesssim &
	\begin{cases}
	(\epsilon^{-2} r^3)
	\epsilon^2\ell^{-2} r^{-1},  \quad   \epsilon^2 \lesssim r\lesssim 2\epsilon^{2/3}    \\
	1 \quad \quad \quad\quad \quad \quad\quad \quad   r\lesssim \epsilon^2.
	\end{cases}    .
	\end{split}
	\end{equation}

On the gluing region $|\vec{x}|\sim \epsilon^{2/3}$, according to  (\ref{asymptoticestimate}) 
\[
\begin{split}
|\nabla^k_H  \log ( H H_0^{-1}   ) |_{\omega'}\lesssim_k ( |X|+|Y|+|Z|^{1/2}    )^{-k} |\vec{X}|^{-1} \max(1, \log ( \frac{|\vec{X}|}{ |X|+|Y|+|Z|^{1/2} }  ) )    .
\end{split}
\]	
Thus the gluing error is small:
\[
|\nabla^k_H  \log ( H h_1^{-1}   ) |_{\omega'}\lesssim_k  ( |X|+|Y|+|Z|^{1/2}    )^{-k} |\vec{X}|^{-1} \max(1, \log ( \frac{|\vec{X}|}{ |X|+|Y|+|Z|^{1/2} }  ) ).
\]
In particular $h_1$ is a Hermitian metric. The deviation of $h_1$ from $H$ causes an error
\[
|F_H-F_{h_1}|_{\omega'} \lesssim  ( |X|+|Y|+|Z|^{1/2}    )^{-2} |\vec{X}|^{-1} \max(1, \log ( \frac{|\vec{X}|}{ |X|+|Y|+|Z|^{1/2} }  ) ),
\]
so the mean curvature error 
\[
\begin{split}
|\Lambda(F_H- F_{h_1}) | \lesssim |F_H-F_{h_1} |_{\omega} \lesssim \epsilon^{-4} |F_H-F_{h_1} |_{\omega'} \\
 \lesssim \epsilon^2 \ell^{-2} r^{-1} \max(1, \log (r\ell^{-1}) ) .
\end{split}
\]
Combining this with (\ref{meancurvatureestimatepreparation}) 
gives the estimate on $|\Lambda F_{h_1}|$. Each differentiation gives an extra factor $O(r^{-1} )$ for $r\lesssim \epsilon^2$, and a factor $O(\ell^{-1})$ for $\epsilon^2\lesssim r\lesssim \epsilon^{2/3}$. 
\end{proof}

\subsection{Behaviour in various characteristic regions}\label{Behaviourinvariouscharacteristicrgions}

We assume $\epsilon^2 e^N\ll 1$, so that $\epsilon r\ll 1$ inside the domain $\mathcal{D}_N$.
The main ansatz $h_1$ exhibits 3 characteristic behaviours in 3 regions:
\begin{itemize}
\item  The region $  \{ r\lesssim \epsilon^{2/3} \}$ where $h_1$ is modelled on $H$.
\item
The region
$\{ |x|+|y|\gtrsim \epsilon r^{1/2} \text{ and } r> 2\epsilon^{2/3} \}$ where $\nabla_{h_1}$ is approximately flat.
\item
The region $\{  |x|+|y|\lesssim \epsilon r^{1/2} \text{ and } r> 2\epsilon^{2/3}   \}$ where the curvature concentrates.
\end{itemize}
The discussion will be very similar to \cite[section 1.1]{Li}.
We introduce a weight function uniformly equivalent to
\begin{equation}\label{rhoregularityscale}
\rho \sim  \begin{cases}
1, \quad  &\ell \gtrsim 1,\\
\ell, \quad   & r\gtrsim \epsilon^2, \quad \ell\lesssim 1,
\\
r, \quad &r\lesssim \epsilon^2.
\end{cases}
\end{equation}

Case I is already summarized in section \ref{monadconstruction} and Lemma \ref{meancurvaturebound2} (\cf also \cite[section 1.1]{Li}), up to a rescaling of the ambient Euclidean metric. The regularity scale in this region is at least $\rho$.

We now examine case II. Without loss of generality consider the subregion 
$|x|\lesssim |y|$, so $|y|\sim \ell\gtrsim \epsilon r^{1/2}$. A basis of holomorphic sections on $E$ can be represented by sections of $\ker \beta$:
\[
s_{(1)}= (0,0,1,0)^t, \quad s_{(2)}= (\sin z/y, 0, 0, 1)^t.
\]
The projections of $s_1,s_2$ to the orthogonal complement of $\text{Im}(\alpha)$ are respectively
\[
\begin{cases}
s_1'= s_{(1)}- \alpha (\alpha^\dag \alpha)^{-1} \alpha^\dag s_{(1)} 
=s_{(1)}- \frac{1}{ (|x|^2+|y|^2)r^{-1} +\epsilon^2        } (x,y,\epsilon^2,0)^t,
\\ 
s_2'= s_{(2)}- \alpha (\alpha^\dag \alpha)^{-1} \alpha^\dag s_{(2)}= s_{(2)}- \frac{\bar{x} \sin z/y}{ |x|^2+|y|^2 + \epsilon^2 r       } (x,y,\epsilon^2,0)^t.
\end{cases}
\]
The Hermitian metric $h_1=h_0$ on the cohomology bundle $E$ is represented by the matrix
\[
h_1(s_{(i)},s_{(j)})= h_{\C^4}( s_i', s_j'  )= \delta_{ij}+ O( \frac{ \epsilon^2 r}{ \ell^2 }   )\approx \delta_{ij}  .
\]
If $\ell \gtrsim 1$, then repeated differentiation shows the higher derivative bound
\[
|\partial^k h_1|\lesssim_k  \epsilon^2 r\ell^{-2}   , \quad k\geq 1,
\]
which is compatible with the bound on $|F_{h_1}|$ in Lemma \ref{meancurvaturebound}. The regularity scale in this region is at least of order $\rho\sim 1$. If instead $\ell\lesssim 1$ and $r\gtrsim 1/3$, then
\[
|\partial^k h_1|\lesssim_k  \epsilon^2 r\ell^{-2-k}   , \quad k\geq 1.
\]
Here taking $z$-derivatives does not significantly change the magnitude of the matrix valued function, while taking $x$ and $y$ derivatives typically results in an additional factor of order $O(\ell^{-1} )$. The regularity scale in this region is at least of order $\rho\sim \ell$. The above higher order estimates hold also for $2\epsilon^{2/3} <r\lesssim 1/3$.

The mean curvature $\Lambda F_{h_1}$ has better estimates for $\ell\lesssim 1$ in region II: we can follow the proof of Lemma \ref{meancurvaturebound} to derive an explict expression for the curvature matrix $( \langle \Lambda F_E s_i', s_j'\rangle )$ like formula (\ref{curvatureansatz}), and notice the cancellation effect as in (\ref{cancellationeffect}). The higher order version of Lemma \ref{meancurvaturebound} for $\ell\lesssim1$ in region II is
\[
|\partial^k \langle \Lambda F_E s_i', s_j'\rangle |\lesssim_k \begin{cases}
\epsilon^2 r \ell^{-2-k}  , \quad & r\gtrsim 1,
\\
\epsilon^2 r^{-1} \ell^{-2-k}, \quad & r\lesssim 1,
\end{cases}      
\]
or equivalently \[
|\nabla^k_{h_1} (\Lambda F_E) |\lesssim_k \begin{cases}
\epsilon^2 r \ell^{-2-k}  , \quad & r\gtrsim 1,
\\
\epsilon^2 r^{-1} \ell^{-2-k}, \quad & r\lesssim 1.
\end{cases}     
\]

We now turn to Case III.  Observe that if the ambient Hermitian structure on $\underline{\C^4}$ is changed from $h_{\C^4}$ to 
\[
\epsilon^{-2}r h_{\C^4}= \text{diag}( 1, 1,  \epsilon^{-2} r, \epsilon^{-2}r   ),
\]
then the induced connection on $E$ is \emph{twisted by a $U(1)$ connection} with curvature $\frac{1}{2}\bar{\partial} \partial \log r$. We shall focus on this twisted situation around a given point $(0,0, \zeta)\in \mathcal{D}_N$ with $|\zeta|> 2\epsilon^{2/3}$. Choose a square root $\sin^{1/2} \zeta$. After rescaling the basis vectors on $\underline{\C^4}$, the twisted monad can be written locally as
\begin{equation}\label{monadnearbubblelocus}
\C\xrightarrow{  ( x,y, \epsilon \sin^{1/2}\zeta ,0     )^t      } \underline{\C^4} \xrightarrow{ (-y,x, 0, \epsilon \sin^{1/2} \zeta) }\underline{\C},
\end{equation}
where the Hermitian structure on $\underline{\C^4}$ is 
\[
\tilde{h}_{ \C^4}=\text{diag}( 1, 1,  \frac{r }{|\sin \zeta| }, \frac{ |\sin \zeta|  r }{ |\sin z|^2 }    ).
\]
For $|\text{Im}(z)|>1$, we have $r=\sin z$, so for $|x|+|y|+|z-\zeta|\lesssim \epsilon |\sin \zeta|^{1/2}$,
\[
\begin{cases}
\tilde{h}_{ \C^4}=\text{diag}( 1, 1,  1+ O(|z-\zeta|) , 1+O(|z-\zeta|   )
\\
\partial^k \tilde{h}_{ \C^4}\lesssim_k 1.
\end{cases}
\]
To leading order, the ambient Hermitian metric on $\underline{\C^4}$ is Euclidean, and the twisted monad (\ref{monadnearbubblelocus}) dimensionally reduces to the monad in the  ADHM construction with parameters $(a_1,a_2,b_1,b_2)=(\epsilon \sin^{1/2}\zeta, 0, 0, \epsilon \sin^{1/2}\zeta )$ (\cf Example \ref{ADHMconstruction}), and the induced connection on the cohomology bundle of the twisted monad is approximately a framed instanton whose moduli parameter is identified as $(\epsilon\sin^{1/2}\zeta,0)\in \C^2/\Z_2$. The regularity scale is $\rho\sim \ell$. One can then estimate in this region the difference between $\nabla_{h_1}$ and the instanton connection $\nabla_{\zeta, \epsilon}$ associated with the ADHM monad, by taking into account both the deviation of $\tilde{h}_{\C^4}$ from being Euclidean, and the effect of $U(1)$-twisting:
\[
|\nabla_{\zeta,\epsilon}^k (F_{h_1}- F_{\nabla_{\zeta,\epsilon} } )|\lesssim_k \ell^{-k}   , \quad | \nabla_{\zeta,\epsilon}^k (\Lambda F_{h_1})|\lesssim_k \ell^{-k}   .
\]
If instead $|\text{Im}(z)|\leq 1$ in region III, then for $|x|+|y|+|z-\zeta|\lesssim \epsilon |\sin \zeta|^{1/2}$,
\[
\begin{cases}
\tilde{h}_{ \C^4}= \text{diag}(1,1,1+O( \frac{\epsilon  +|z-\zeta| }{ |\sin \zeta| }  ),1+O( \frac{\epsilon +|z-\zeta|}{ |\sin \zeta| }  )   ),
\\
|\partial^k \partial_x \tilde{h}_{ \C^4}|+ |\partial^k \partial_y \tilde{h}_{ \C^4}|  \lesssim_k \frac{  \epsilon }{ |\sin \zeta|^{3/2} \ell^k }  , \quad k\geq 0.
\\
|\partial^k \partial_z \tilde{h}_{ \C^4} |\lesssim_k \frac{1  }{ |\sin \zeta| \ell^{k} }   ,  \quad  |\partial^k \partial_z \partial_{\bar{z}} \tilde{h}_{ \C^4} |\lesssim_k \frac{1}{ |\sin \zeta|^{2} \ell^{k} }, \quad k\geq 0.
\end{cases}
\]
The connection is still locally modelled on the ADHM one-instantons, albeit the larger deviation. The regularity scale is $\rho\sim \ell$. The curvature deviation has estimates
\[
|\nabla_{\zeta,\epsilon}^k (F_{h_1}- F_{\nabla_\zeta} )|\lesssim_k |\sin \zeta|^{-2} \ell^{-k}       , \quad | \nabla_{\zeta,\epsilon}^k (\Lambda F_{h_1})|\lesssim_k |\sin \zeta|^{-2} \ell^{-k} . 
\]

Combining the above, we have a unified higher order estimate for $\Lambda F$:

\begin{cor}\label{meancurvaturehigherorder}
	In the region $r> 2\epsilon^{2/3}$, 
	\[
	|\nabla^k_{h_1} (\Lambda F_{h_1}) |\lesssim_k  \begin{cases}
	\epsilon^2 r \ell^{-2} \rho^{-k}, \quad & r\gtrsim 1, \\
	\epsilon^2 r^{-1} \ell^{-2} \rho^{-k}, \quad  & r\lesssim 1.
	\end{cases}	
	\]
\end{cor}

\section{Dirichlet problem}

We shall construct our examples of HYM metrics by solving a Dirichlet boundary value problem. Our approach is a computational shortcut, but the alternative method based on perturbative analysis seems likely to work as well.

The following general result is proved by combining Donaldson's solution to the Dirichlet problem in the smooth case \cite{Donaldsonboundary}, with Bando and Siu's method to treat reflexive sheaves \cite{BandoSiu}.

\begin{thm}\cite{Li}
Let $E$ be a reflexive sheaf over a compact K\"ahler manifold $(\overline{Z}, \omega)$ with nonempty boundary $\partial Z$, which is locally free near the boundary. For any Hermitian metric $f$ on the restriction of $E$ to $\partial Z$ there is a unique Hermitian $h$ on $E$, which is smooth on the locally free locus, has finite $L^2$ curvature, and solves
\[
\sqrt{-1}\Lambda F_h=0 \text{ in } Z, \quad H=f \text{ over } \partial Z.
\]
\end{thm}

In our situation, we apply the theorem to the reflexive sheaf $E$ over the domain $\mathcal{D}_N$, with boundary data $h_1$. This provides a HYM metric $h=h_{N,\epsilon}$ depending on $\epsilon$ and $N$. The rest of this section is to show that geometric properties of $h_1$ transfer to $h$.

\subsection{Barrier construction}

Consider the Newtonian potential on $\C^3=\R^6$, given up to constant by
\[
\gamma(\vec{x})=\gamma(x,y,z)= \frac{1}{ |\vec{x}|^4  }= \frac{ 1}{ |(x,y,z)|^4   }.
\]
The periodic Newtonian potential on $\C^3/\Z$ is then
\begin{equation}
\Gamma(\vec{x})= \sum_{n=-\infty}^\infty \frac{ 1}{ |(x,y,z+2\pi n)|^4   }.
\end{equation}
Without loss of generality  $|\text{Re}(z)|\leq \pi$.
Using Cauchy's integral test, 
\begin{equation}
\Gamma(\vec{x}) \lesssim \begin{cases}
\frac{ 1}{ |\vec{x}|^4   }, \quad &|\vec{x}|< \frac{1}{2},
\\
\frac{1}{  |\vec{x}|^3    }, \quad & r\gtrsim  \frac{1}{2}.
\end{cases}
\end{equation}

We consider the function on $\mathcal{D}_N$,
\begin{equation}
U(\vec{x})= \int_{\mathcal{D}(N) } |\Lambda F_{h_1} |(\vec{x}') \Gamma(\vec{x}-\vec{x}') d\text{Vol}(\vec{x}')
\end{equation}

\begin{lem}\label{boundonU}
Assume $\epsilon^2 e^N\ll 1$. Then 
\begin{equation}
\begin{split}
U(\vec{x}) \lesssim \begin{cases}
\epsilon^2 e^N \max( \log N, -\log \ell  ), \quad  &r\gtrsim 1/10,
\\
\epsilon^2 e^N \log N+ \epsilon^2 r^{-1} \max(1, \log \frac{r}{\ell}), \quad & 4\epsilon^{2/3} \leq r< 1/10,
\\
\epsilon^2 e^N \log N+ \epsilon^2 r^{-1} \max(1, \log^2 \frac{r}{\ell}), \quad & \frac{\epsilon^{2/3}}{2} \leq r< 4\epsilon^{2/3},
\\
\epsilon^2 e^N \log N+  \epsilon^{4/3}+ r^2 \log \frac{r}{\ell} , \quad &r< \frac{\epsilon^{2/3} }{2}.
\end{cases}
\end{split}
\end{equation}
\end{lem}

\begin{proof}
The starting point is the pointwise upper bound on $|\Lambda F_{h_1}|$ from Lemma \ref{meancurvaturebound} and \ref{meancurvaturebound2}. We divide and conquer the contributions from various regions. Notice by assumption $|\epsilon^2 \sin z' |\lesssim \epsilon^2 e^N \ll 1$.

We first analyse the source from $\{ r(\vec{x}')\gtrsim \frac{1}{3}   \}$, where $|\Lambda F_{h_1}|\lesssim \epsilon^2 r\ell^{-2}$. For $|z-z'|\gtrsim 1 $, the contribution from the 4-plane
\[
\begin{split}
&\int_{\R^4} \ell^{-2}(x',y', z') \Gamma(\vec{x}-\vec{x}') d\mathcal{H}^4(x',y') \\
\lesssim &
\int_{\R^4}  (|x'|^2+|y'|^2+\epsilon^2 |\sin z'|)^{-1} |\vec{x}-\vec{x}'|^{-3}d\mathcal{H}^4(x',y') 
\\
\lesssim & |z-z'|^{-1}.
\end{split}
\]
Similarly, for $|z-z'|\lesssim 1$,
\[
\begin{split}
\int_{\R^4\setminus \{ |x-x'|+|y-y'|\lesssim 1   \}  } \ell^{-2}(x',y', z') \Gamma(\vec{x}-\vec{x}') d\mathcal{H}^4(x',y')
\lesssim 1.
\end{split}
\]
Integrating over all such 4-planes in the $\vec{x}'$-domain,
\[
\begin{split}
&\int_{ \mathcal{D}_N\setminus \{ |\vec{x}-\vec{x}'|<1/6  \}   }  r\ell^{-2}(x',y', z') \Gamma(\vec{x}-\vec{x}') dvol(\vec{x}')
\\
\lesssim &
\int_{ \{ |\text{Im}(z')|<N, |\text{Re}(z')|\leq \pi  \}   }  |z-z'|^{-1} (1+ |\sin z'|)  d\text{Re}(z')d\text{Im}(z')
\\
\lesssim &  e^{N}\log N.
\end{split}
\]
On the other hand, consider the short distance contribution from $\{ |\vec{x}'-\vec{x}|< 1/6 \}$. If $r(\vec{x})\gtrsim 1/10$, then
\[
\begin{split}
&\int_{ |\vec{x}-\vec{x}'|<1/6     } r\ell^{-2}(x',y', z') \Gamma(\vec{x}-\vec{x}') dvol(\vec{x}') \\
\lesssim 
&
e^N \int_{ |\vec{x}-\vec{x}'|<1/6  } (|x'|^2+|y'|^2+\epsilon^2 |\sin z'|)^{-1} |\vec{x}-\vec{x}'|^{-4} dvol(\vec{x}')
\\
\lesssim & 
 e^N \max(1,  -\log \ell(\vec{x}) ).
\end{split}
\]
Here the last inequality can be seen by summing up contributions from all dyadic scales $|\vec{x}-\vec{x}'|\sim 2^{-k}$. If $|\vec{x}|<  \frac{1}{10}$ instead, then $\vec{x}$ is sufficiently far from $\{  |\vec{x}'|\gtrsim 1/3  \}$, so that this short distance contribution is zero. To summarize, the contribution to $U(\vec{x})$ from the region $\{  r(\vec{x}')\gtrsim 1/3  \}$ is bounded by \[
\begin{cases}
O( \epsilon^2 e^N \max( \log N, -\log \ell(\vec{x}) ) ), \quad  & r\gtrsim 1/10,
\\
O( \epsilon^2 e^N \log N), \quad & r<1/10  .
\end{cases}
\]

Next we analyse the source from $\{    2\epsilon^{2/3} < |\vec{x}'|< 1/3 \}$, which admits the bound $|\Lambda F_{h_1}| \lesssim \epsilon^2 \ell^{-2}r^{-1} $. Since this source has total $L^1$ integral $O(\epsilon^2)$, for $r\gtrsim 1/2$ the contribution to $U(\vec{x})$ is bounded by $O(\epsilon^2 r^{-3})$, which is negligible compared to the previous case. We focus on $\{ |\vec{x}|<1/2  \}$, where the contribution is bounded by
\[
\epsilon^2\int_{ \{   2 \epsilon^{2/5} <|\vec{x}'|< 1/3 \} } \frac{1}{ (|x'|+|y'|+\epsilon|z'|^{1/2} )^2 |\vec{x}'| } \frac{1}{ |\vec{x}-\vec{x}'|^4  } dvol(\vec{x}').
\]
This integral can be estimated by summing up the contribution from all dyadic scales $ |\vec{x}-\vec{x}'|\sim 2^{-k} $. For $\epsilon^{2/3} \lesssim |\vec{x}|\lesssim 1/2$, the integral is bounded by $O( \epsilon^2 r^{-1} \log \frac{r}{\ell} )$. For $|\vec{x}|< \epsilon^{2/3}$, which is separated from the support of the source, the contribution is bounded by $O(\epsilon^2 \epsilon^{-2/3} )=O(\epsilon^{4/3})$.

The source from the gluing region $\{ \epsilon^{2/3} \leq |\vec{x}'|\leq 2\epsilon^{2/3}   \}$ is very similar except for an extra logarithmic factor. As in the above case, the effect of this source is local. In the region $\{ \frac{\epsilon^{2/3}}{2} < |\vec{x}|<4\epsilon^{2/3}   \}$, the contribution is bounded by $O( \epsilon^2 r^{-1} (\log \frac{r}{\ell})^2  )$, and the other regions are essentially unaffected.

Finally we analyse the source from $\{ |\vec{x}'|\leq \epsilon^{2/3}  \}$ where 
\[
|\Lambda F_{h_1}|\lesssim  \begin{cases}
1, \quad   & |\vec{x}'|\lesssim  \epsilon^2,
\\
 r^2   \ell^{-2}, \quad   & \epsilon^2< |\vec{x}'|\leq \epsilon^{2/3} .
\end{cases}
\]
As in the above case, the effect on $U(\vec{x})$ is essentially localized to $\{ |\vec{x}|<4\epsilon^{2/3}  \}$.
For $\epsilon^2\lesssim |\vec{x}|\leq 4\epsilon^{2/5}$, this integral can be decomposed into contribution from $|\vec{x}|\lesssim |\vec{x}'|\lesssim 2|\vec{x}|$, which is $O(|\vec{x}|^2 \log \frac{r}{\ell} )$, and the contribution from elsewhere, which is
 $O(\epsilon^{4/3}   )$. For $|\vec{x}|\lesssim \epsilon^2$, this integral is $O( \epsilon^{4/3} )$.

Combining the above yields the result.
\end{proof}

\subsection{Metric deviation and curvature estimates}\label{Metricdeviationandcurvatureestiamtes}

For $\epsilon$ sufficiently small depending on $N$,
we shall show that the HYM metric $h_{N,\epsilon}$ is close to the main ansatz $h_1$, so inherits its geometric properties.

\begin{thm}\label{metricdeviationtheorem}
	Assume $\epsilon\ll e^{-N}$.
	The HYM metric $h$ on $E$ admits the estimates:
	\begin{equation}\label{metricdeviation}
	|\nabla^k_{h_1}  \log ( h h_1^{-1}   ) |\lesssim_k \rho^{-k} \begin{cases}
 \epsilon^2 e^N \max( \log N, -\log \ell  ), \quad  &r\gtrsim 1/10,
	\\
	\epsilon^2 e^N \log N+ \epsilon^2 r^{-1} \max(1, \log \frac{r}{\ell}), \quad & 4\epsilon^{2/3} \leq r< 1/10,
	\\
	\epsilon^2 e^N \log N+ \epsilon^2 r^{-1} \max(1, \log^2 \frac{r}{\ell}), \quad & \frac{\epsilon^{2/3}}{2} \leq r< 4\epsilon^{2/3},
	\\
	\epsilon^2 e^N \log N+  \epsilon^{4/3}+ r^2 \log \frac{r}{\ell} , \quad &r< \frac{\epsilon^{2/3} }{2}.
	\end{cases}
	\end{equation}
Here the weight $\rho$ is given by (\ref{rhoregularityscale}). All estimates are uniform in $N$ and $\epsilon$.
\end{thm}

\begin{proof}
We need the almost subharmonicity estimate  (\cf Lemma 2.5 in \cite{Ni}):
	\begin{equation}\label{subharmonicity}
	\begin{cases}
	\Lap \log  \text{Tr} (h h_1^{-1}) \geq -C(|\Lambda F_{h_1}|+ |\Lambda F_{h}|)= -C|\Lambda F_{h_1}|,
	\\
	\Lap \log \text{Tr} (h_1 h^{-1}) \geq - C|\Lambda F_{h_1}| ,
	\end{cases}
	\end{equation}
	Notice near the origin (\ref{subharmonicity}) continues to hold in the distributional sense, since the $L^2$ curvature of $h$ and $h_1$ are finite, and the singularity has complex codimension 3 (\cf proof of Proposition 3.1 in \cite{JacobWalpuski2}). Using the boundary condition $\text{Tr} (h h_1^{-1})= \text{Tr} (h_1 h^{-1})=\text{rank}(E)=2$, by the comparison principle 
	\[
	\log  \frac{\text{Tr} (h h_1^{-1})}{2} \geq -CU, \quad  \log \frac{\text{Tr} (h_1 h^{-1})}{2} \geq -CU,
	\]
	or equivalently there is a \emph{pointwise estimate} on $\mathcal{D}_N$:
	\begin{equation}
	e^{-CU}  \leq h\leq e^{CU} h_1, 
	\end{equation}
Notice by Lemma \ref{boundonU}, the function $U\ll 1$ on $\mathcal{D}_N$, so in fact
\[
|h-h_1| \leq CU h_1 \ll h_1.
\]

The regularity scale of $h_1$ is bounded below by $\rho$.	
Recall also the higher order control on the mean curvature of the ansatz in Lemma \ref{meancurvaturebound2} and \ref{meancurvaturehigherorder}. Applying Bando and Siu's interior regularity estimate (\cf Appendix C, D in \cite{JacobWalpuksi}) to rescaled balls, we derive the higher order estimate (\ref{metricdeviation}) as required.
\end{proof}

The geometric properties of $h_1$ from section \ref{Behaviourinvariouscharacteristicrgions} carries over to $h_{N,\epsilon}$:

\begin{cor}
(Transverse bubbles)
For every nonzero $(0,0,\zeta)\in \mathcal{D}_N$, let $\lambda_{\zeta, \epsilon}$ be the scaling diffeomorphism $\vec{x}\mapsto (0,0,\zeta)+ \epsilon\vec{x}$. Then for every fixed $N$ and $\zeta$, the rescaled HYM connections $\lambda_{\zeta,\epsilon}^* \nabla_{h_{N,\epsilon} } $ converge in $C^\infty_{loc}$ up to gauge as $\epsilon\to 0$ to a HYM connection on $\C^3$, which is isomorphic to the pullback of the ADHM one-instanton with parameter $(\sin^{1/2} \zeta, 0)\in \C^2/\Z_2$ from Example \ref{ADHMconstruction}.

\end{cor}

\begin{rmk}
Geometrically, we say the transverse bubbles along the $z$-axis are described by the \emph{Fueter map} into the moduli space of framed instantons:
	\begin{equation}
	\C\to \C^2/\Z_2, \quad \zeta\mapsto (\sin^{1/2}\zeta, 0),
	\end{equation}
	which is well defined independent of the choice of square root. 
\end{rmk}

\begin{cor}
(High codimensional bubble at the origin) Let $\lambda_{0,\epsilon^2}$ be the scaling diffeomorphism $\vec{x}\mapsto \epsilon^2 \vec{x}$. Then for every fixed $N$, as $\epsilon\to 0$, the rescaled HYM connections $\lambda_{0,\epsilon^2 }^*\nabla_{ h_{N,\epsilon} }$ converge in $C^\infty_{loc}(\C^3\setminus \{0\})$ up to gauge to the model HYM connection $\nabla_H$ introduced in section \ref{monadconstruction}.
\end{cor}

\begin{rmk}
Notice instead that the \emph{tangent cone at the origin} for $h_{N,\epsilon}$ is the pullback of the Levi-Civita connection on $\mathbb{CP}^2$ up to twisting by a $U(1)$-connection. To see the tangent cone requires a higher magnification rate around the origin. The model HYM metric $H$ facilitates the transition behaviour from the tangent cone at the origin to the transverse bubbles along the $z$-axis. 
\end{rmk}

\begin{cor}
(uniform $L^2$-energy bound) Assume $\epsilon\ll e^{-N}$. The curvature of the HYM connections $h=h_{N,\epsilon}$ satisfy 
\begin{equation}\label{L2energybound}
\int_{\mathcal{D}_N  }  |F_{h_{N,\epsilon}} |^2 dvol \leq CN,
\end{equation}
where the constant is independent of $N$ and $\epsilon$. For fixed $N$, as $\epsilon\to 0$, the curvature density measure converges weakly to a measure supported on $\{ x=y=0 \}$,
\[
|F_{h_{N,\epsilon}} |^2 dvol \to \Theta \mathcal{H}^2 \lfloor \{  x=y=0 \},
\]
where $\Theta$ is a constant equal to the $L^2$-energy of the one-instanton on $\R^4$.
\end{cor}

\begin{proof}
From Theorem \ref{metricdeviationtheorem} the deviation of curvature has the estimate
\[
|F_h- F_{h_1} |\lesssim \rho^{-2}\begin{cases}
\epsilon^2 e^N \max( \log N, -\log \ell  ), \quad  &r\gtrsim 1/10,
\\
\epsilon^2 e^N \log N+ \epsilon^2 r^{-1} \max(1, \log \frac{r}{\ell}), \quad & 4\epsilon^{2/3} \leq r< 1/10,
\\
\epsilon^2 e^N \log N+ \epsilon^2 r^{-1} \max(1, \log^2 \frac{r}{\ell}), \quad & \frac{\epsilon^{2/3}}{2} \leq r< 4\epsilon^{2/3},
\\
\epsilon^2 e^N \log N+  \epsilon^{4/3}+ r^2 \log \frac{r}{\ell} , \quad &r< \frac{\epsilon^{2/3} }{2}.
\end{cases}
\]
Hence
\[
\norm{F_h-F_{h_1} }_{L^2 } \lesssim \epsilon^2 e^N  \log^3 \frac{1}{\epsilon},
\]
which is an exponentially small quantity vanishing in the $\epsilon\to 0$ limit. But the curvature of $h_1$ is estimated in Lemma \ref{meancurvaturebound} and \ref{meancurvaturebound2}:
\[
|F_{h_1} |\lesssim \begin{cases}
\epsilon^2 r \ell^{-2} ,\quad  & \ell\gtrsim 1,
\\
\epsilon^2 r \ell^{-4}, \quad & \ell \lesssim 1, \quad r\gtrsim \epsilon^2,
\\
r^{-2} ,\quad & r\lesssim \epsilon^2.
\end{cases}
\]
The essential contribution to the $L^2$-curvature comes from $\{ |x|+|y|\lesssim \epsilon r^{1/2} \}$.
For each $\zeta\in \C$ with $|\text{Im}(\zeta)|<N$ and $|\text{Re}(\zeta)|\leq \pi$, the integral on the 4-plane
\[
\int_{ \{ z= \zeta, |x|+|y|<N \}  \setminus \{ r\lesssim \epsilon^2  \} } |F_h|^2(x,y,z) d\mathcal{H}^4(x,y) \leq C,
\]
so integrating over all such 4-planes, and taking into account the small region $\{ r\lesssim \epsilon^2 \}$,
\[
\int_{\mathcal{D}_N } |F_h|^2 \leq C N.
\]

The measure convergence is because along the $z$-axis and slightly away from the singular points, the HYM connection is modelled on ADHM one-instantons, and the curvature decays rapidly away from the $z$-axis.
\end{proof}

To relate this discussion back to the setting of Theorem \ref{introthm}, we simply rescale the domain $\mathcal{D}_N$ by a factor $N^{-1}$, and pass to the covering space. This produces a family of HYM metrics on $B_1\subset \C^3$, still denoted as $h_{N,\epsilon}$, which are periodic with respect to the translation $z\mapsto z+2\pi N^{-1}$. Because of the scaling of coordinates, for any fixed $N$ the transverse bubbles  along the $z$-axis as $\epsilon\to 0$ are now described by the Fueter map
\[
\C \mapsto \C^2/\Z_2, \quad \zeta\mapsto (\sin^{1/2}(N\zeta), 0  ),
\]
whose \emph{modulus of continuity has no uniform bound} in $N$. The singular points correspond to the zeros of the HYM connection, so the number grows as $O(N)$.

Scaling down $\mathcal{D}_N$ causes the $L^2$-energy to shrink by a factor $N^{-2}$. Taking $O(N)$ periodic copies of the construction by translation in the $\text{Re}(z)$-variable, we finally get the \emph{uniform $L^2$-energy} estimate independent of $N$ and $\epsilon$:
\[
\int_{B_1} |F|^2 \leq C.
\]
Thus Theorem \ref{introthm} is proved.

\end{document}